\documentclass[a4paper,12pt,reqno]{amsart}
\usepackage[utf8]{inputenc}
\usepackage{amsmath}
\usepackage{amsthm,amssymb,amsfonts}
\usepackage{graphicx}
\usepackage{hyperref}
\usepackage{url}

\usepackage{geometry}
\usepackage{array,bm}
\usepackage{color}
\usepackage{verbatim}

\newtheorem{theorem}{Theorem}
\newtheorem{definition}[theorem]{Definition}

\newtheorem{remark}[theorem]{Remark}
\newtheorem*{question}{Question}

\makeatletter
\@namedef{subjclassname@2020}{\textup{2020} Mathematics Subject Classification}
\makeatother

\begin{document}
\title{On $L^p$-Hardy inequalities for magnetic $p$-Laplacians}

\author{Yi C. Huang} 
\address{School of Mathematical Sciences, Nanjing Normal University, Nanjing 210023, People's Republic of China}
\email{Yi.Huang.Analysis@gmail.com}
\urladdr{https://orcid.org/0000-0002-1297-7674}

\author{Xinhang Tong}
\address{School of Mathematical Sciences, Nanjing Normal University, Nanjing 210023, People's Republic of China}
\email{letterwoodtxh@gmail.com}

\date{\today} 


\thanks{Research of YCH is partially supported by the National NSF grant of China (no. 11801274), 
the JSPS Invitational Fellowship for Research in Japan (no. S24040), and the Open Projects from Yunnan Normal University (no. YNNUMA2403) and Soochow University (no. SDGC2418). 
YCH thanks the 2025 ISAAC travelling grant under which he was exposed to the subject studied here.}

\begin{abstract}
In this paper we revisit the remainder terms of $L^p$-Hardy inequalities for magnetic $p$-Laplacians. 
In particular, we will give an integral representation of the sharp constant for a crucial algebraic inequality established by C. Cazacu, D. Krejčiřík, N. Lam, and A. Laptev.

\end{abstract}

\maketitle

\section{Introduction}
The following well-known Euclidean $L^p$-Hardy inequalities (see \cite{hardy1952} for example) hold in any dimension $d\ge 2$ and for every $1<p<d$:
$$
\int_{\mathbb{R}^d}|\nabla u|^{p} dx\ge \mu_{p,d}\int_{\mathbb{R}^d}\frac{|u|^p}{|x|^p} dx, \quad \forall\, u\in W^{1,p}(\mathbb{R}^d),
$$
where the sharp constant $\mu_{p,d}$ is equal to $(\frac{d-p}{p})^p$. Recently there is a rising interest in the improved Hardy inequality for the \textit{magnetic $p$-Laplacian} which we now recall.  

Let $B:\mathbb{R}^d\to\mathbb{R}^{d\times d}$ be a smooth matrix-valued function representing the exact magnetic field, then there exists a smooth magnetic potential $A:\mathbb{R}^d\to\mathbb{R}^{d}$, such that $dA=B$ (see for example \cite{Cazacu02072016} and \cite[Section 1.2]{Cazacu2024} for more details).
\begin{definition}
The magnetic p-Laplacian is formally defined on $C_{c}^{\infty}(\mathbb{R}^{d})$ by
$$
\Delta_{A,p}u:=\mathrm{div}_{A}(|\nabla_{A}u|^{p-2}\nabla_{A}u),
$$
where the magnetic gradient and magnetic divergence are given by
$$
\nabla_{A}u:=\nabla u+iA(x)u, \quad \mathrm{div}_{A}F:=\mathrm{div}F+iA\cdot F.
$$
Here, $1<p<\infty$ and $F:\mathbb{R}^{d}\to\mathbb{C}^{d}$ denotes a smooth vector field.
\end{definition}

The associated quadratic form $h_{A,p}$ of the magnetic $p$-Laplacian $\Delta_{A,p}$, with its form domain denoted by $\mathcal{D}(h_{A,p})$, is then defined by
$$
h_{A, p}[u]:=\int_{\mathbb{R}^{d}}|\nabla_{A}u|^{p} dx, \quad  u\in\mathcal{D}(h_{A,p}):=\overline{C_{c}^{\infty}(\mathbb{R}^{d})}^{\left\| \cdot \right\|},
$$
where the norm $\left\| \cdot \right\|$ with respect to which the closure is taken is given by
$$
\left\| u \right\|:=\left(h_{A, p}[u]+\left\| u \right\|_{L^{p}(\mathbb{R}^d)}^{p}\right)^{1/p}.
$$

In a recent interesting paper, C. Cazacu, D. Krejčiřík, N. Lam, and A. Laptev established for the magnetic p-Laplacians the following improved Hardy inequalities.
\begin{theorem}[\cite{Cazacu2024}, Theorem 1.2]\label{CKLLTH}
Let $2 \leq p < d$ and $B$ be a smooth and closed magnetic field with $B \neq 0$. 
Then there exists a constant $c_{p} > 0$ such that for any vector field $A$ with $dA = B$, and for any $u \in \mathcal D(h_{A,p})$,
\begin{equation}\label{MHI}
\int_{\mathbb{R}^d} |\nabla_A u|^p  dx - \mu_{p,d} \int_{\mathbb{R}^d} \frac{|u|^p}{|x|^p}  dx \geq c_{p}\int_{\mathbb{R}^d} \left| \nabla_A \left( u |x|^{\frac{d-p}{p}} \right) \right|^p |x|^{p-d}  dx.
\end{equation}
The constant $c_p$ is given by
\begin{equation}\label{infts}
c_p = \inf_{(s,t) \in \mathbb{R}^2 \setminus \{(0,0)\}} \frac{[t^2 + s^2 + 2s + 1]^{\frac{p}{2}} - 1 - ps}{[t^2 + s^2]^{\frac{p}{2}}} \in (0,1].
\end{equation}
\end{theorem}

The “improved” Hardy inequality \eqref{MHI} also holds when $p\in(1,2)$ with the same constant $c_{p}$ in \eqref{infts}, however, $c_p\equiv 0$ in this case \textup{(see \cite[Proposition 3.1]{Cazacu2024})}.

Note that in \cite{Cazacu2024}, \eqref{MHI} was reduced to the following minimisation problem.

\begin{question}
What is the sharp constant $c_{p}$ for the following algebraic inequality:
\begin{equation}\label{mincp}
C_{p}(x, y) := |x|^{p} - |x - y|^{p} - p |x - y|^{p-2} \Re((x - y) \cdot \overline{y} )\geq c_{p}|y|^{p}, \quad \forall\,x, y \in \mathbb{C}^d,
\end{equation}
where $\Re$ stands for the real part of a complex number?
\end{question}

Inspired by the work \cite{IIO2017} of N. Ioku, M. Ishiwata and T. Ozawa, 
we propose an integral representation of $C_{p}(x,y)$ which allows us to revisit the sharp $c_{p}$ in \eqref{mincp}. 
Trivially, \eqref{mincp} holds for $y=0$, $x=0$ or $y=x$. Also, $c_{p}=1$ for $p=2$. Without loss of generality, we shall not analyze these trivial cases in what follows.

Firstly, for $\xi, \eta\in\mathbb{C}^{d}$ we introduce the following quantity
$$
R_p(\xi, \eta) =
\begin{cases}
\displaystyle
\frac{1}{|\xi - \eta|^2} \left( \frac{1}{p} |\eta|^p + \frac{1}{p'} |\xi|^p - |\xi|^{p-2}\Re( \xi \cdot \bar{\eta}) \right), & \text{if } \xi \ne \eta, \\
\\
\displaystyle
\frac{p - 1}{2} |\xi|^{p - 2}, & \text{if } \xi = \eta,
\end{cases}
$$
where $p>1$ and $\frac{1}{p'}=1-\frac{1}{p}$. 

Our first result is stated as follows.
\begin{theorem}\label{exrep}
We have the following integral representation: for any $\xi, \eta\in\mathbb{C}^{d}$,
\begin{equation}\label{eqrep}
\begin{aligned}
R_{p}(\xi, \eta)
  &=(p-1)\int_{0}^{1}|\theta\xi+(1-\theta)\eta|^{p-2}\theta d\theta\\
  &\quad+(p-2)\int_{0}^{1}|\theta\xi+(1-\theta)\eta|^{p-4}\frac{|\xi|^2|\eta|^{2}-\Re^{2}(\xi\cdot\bar{\eta})}{|\xi-\eta|^{2}}\theta d\theta,
\end{aligned}
\end{equation}
and if $p\geq2$,
$$R_{p}(\xi, \eta)\ge (p-1)\int_{0}^{1}|\theta\xi+(1-\theta)\eta|^{p-2}\theta d\theta.$$
The above inequality becomes an equality when $\xi=\lambda\eta$ with $\lambda\in \mathbb R$.
\end{theorem}

\begin{remark}
The case $\xi, \eta\in\mathbb{R}$ was considered in \cite{IIO2017}, in particular,
$$
R_{p}(\xi, \eta)=(p-1)\int_{0}^{1}|\theta\xi+(1-\theta)\eta|^{p-2}\theta d\theta.
$$
\end{remark}

\begin{remark} \label{rem:lambda}
Using Theorem \ref{exrep} and setting $\xi=x-y$, $\eta=x$, we have
\begin{equation}\label{estimation}
\begin{aligned}
\frac{C_{p}(x,y)}{|y|^p}
&=p\frac{R_{p}(x-y, x)}{|y|^{p-2}}\\
&=p(p-1)\int_{0}^{1}\left|\theta \frac{y}{|y|}-\frac{x}{|y|}\right|^{p-2}\theta d\theta\\
&\quad+(p-2)\int_{0}^{1}|\theta y-x|^{p-4}\frac{|x-y|^2|x|^{2}-\Re^{2}((x-y)\cdot\bar{x})}{|y|^{p}}\theta d\theta,
\end{aligned}
\end{equation}
and if $p\geq2$,
$$\frac{C_{p}(x,y)}{|y|^p}\ge p(p-1)\int_{0}^{1}\left|\theta \frac{y}{|y|}-\frac{x}{|y|}\right|^{p-2}\theta d\theta.$$
The above inequality becomes an equality when $x=\lambda y$ with $\lambda\in \mathbb R$.
\end{remark}

Our second result is stated as follows.

\begin{theorem}\label{solution}
i) Let $p\ge 2$. The sharp constant $c_p$ in \eqref{mincp} is determined as 
$$
c_{p}=(p-1)(1-k_{0})^p+pk_{0}(1-k_{0})^{p-1}+k_{0}^p>0,
$$
in the case of $x=k_{0}y$, where $k_{0}=\frac{r_{0}}{1+r_{0}}$ and $r_{0}$ is the solution of the equation
$$
r^{p-1}-(p-1)r-(p-2)=0.
$$
ii) Let $1<p< 2$. The sharp constant $c_p$ in \eqref{mincp} is equal to 0.
\end{theorem}

\begin{remark}
Let $p=3$. We obtain $r_{0}=\sqrt{2}+1$, $k_{0}=\frac{\sqrt{2}}{2}$, and
$$
c_{3}=2\left(1-\frac{\sqrt{2}}{2}\right)^3+3\frac{\sqrt{2}}{2}\left(1-\frac{\sqrt{2}}{2}\right)^2+\left(\frac{\sqrt{2}}{2}\right)^3=2-\sqrt{2}.
$$
\end{remark}

\section{Proofs of Theorem \ref{exrep} and Theorem \ref{solution}}
\begin{proof}[Proof of Theorem \ref{exrep}]
Trivially, \eqref{eqrep} holds for $\xi=\eta$. For $\xi\neq\eta$, we obtain
\begin{equation*}
\begin{aligned}
  &\frac{d}{d\theta}|\theta\xi+(1-\theta)\eta|^p=p|\theta\xi+(1-\theta)\eta|^{p-2}\Re\left((\theta\xi+(1-\theta)\eta)\cdot\overline{(\xi-\eta)}\right)
\end{aligned}
\end{equation*}
and
\begin{equation*}
\begin{aligned}
  &\frac{d}{d\theta}|\theta\xi+(1-\theta)\eta|^{p-2}\Re\left((\theta\xi+(1-\theta)\eta)\cdot\bar{\eta}\right)\\
  &\qquad=(p-2)|\theta\xi+(1-\theta)\eta|^{p-4}\left(\theta|\xi|^2+(\theta-1)|\eta|^2+(1-2\theta)\Re(\xi\cdot\bar{\eta})\right)\\
  &\qquad\qquad\times\left(\theta\Re(\xi\cdot\bar{\eta})+(1-\theta)|\eta|^2\right)\\
  &\qquad\quad+|\theta\xi+(1-\theta)\eta|^{p-2}\Re\left((\xi-\eta)\cdot\bar{\eta}\right).
\end{aligned}
\end{equation*}
By a direct calculation, we obtain
\begin{equation*}
\begin{aligned}
&\left(\theta|\xi|^2+(\theta-1)|\eta|^2+(1-2\theta)\Re(\xi\cdot\bar{\eta})\right)\left(\theta\Re(\xi\cdot\bar{\eta})+(1-\theta)|\eta|^2\right)\\
&\qquad=\theta^2|\xi|^2\Re(\xi\cdot\bar{\eta})+(\theta-1)\theta|\eta|^2\Re(\xi\cdot\bar{\eta})+\theta(1-2\theta)\Re^{2}(\xi\cdot\bar{\eta})\\
&\qquad\quad+\theta(1-\theta)|\xi|^2|\eta|^2-(1-\theta)^2|\eta|^2|\eta|^2+(1-2\theta)(1-\theta)\Re(\xi\cdot\bar{\eta})|\eta|^2\\
&\qquad=|\theta\xi+(1-\theta)\eta|^{2}\Re((\xi-\eta)\cdot\bar{\eta})+\theta|\xi|^2|\eta|^{2}-\theta\Re^{2}(\xi\cdot\bar{\eta}).
\end{aligned}
\end{equation*}
Since
$$
[|\theta\xi+(1-\theta)\eta|^p]_{\theta=0}^{1}=|\xi|^{p}-|\eta|^{p}
$$
and
$$
[|\theta\xi+(1-\theta)\eta|^{p-2}\Re((\theta\xi+(1-\theta)\eta)\cdot\bar{\eta})]_{\theta=0}^{1}=|\xi|^{p-2}\Re(\xi\cdot \bar{\eta})-|\eta|^{p-2}\Re( \eta\cdot \bar{\eta}),
$$
we obtain
\begin{equation*}
\begin{aligned}
  &\frac{1}{p} |\eta|^p + \frac{1}{p'}  |\xi|^p - |\xi|^{p-2} \Re(\xi\cdot\bar{\eta}) \\
  &\qquad=\left(1-\frac{1}{p} \right) \left( |\xi|^p - |\eta|^p \right)-\left(|\xi|^{p-2}\Re(\xi\cdot \bar{\eta})- |\eta|^{p-2}\Re( \eta\cdot \bar{\eta})\right)\\
  &\qquad=(p-1)\int_{0}^{1}|\theta\xi+(1-\theta)\eta|^{p-2}\Re((\theta\xi+(1-\theta)\eta)\cdot\overline{(\xi-\eta)}) d\theta\\
  &\qquad\quad-(p-1)\int_{0}^{1}|\theta\xi+(1-\theta)\eta|^{p-2}\Re\left((\xi-\eta)\cdot\bar{\eta}\right) d\theta\\
  &\qquad\quad+(p-2)\int_{0}^{1}|\theta\xi+(1-\theta)\eta|^{p-4}(|\xi|^2|\eta|^{2}-\Re^{2}(\xi\cdot\bar{\eta}))\theta d\theta\\
  &\qquad=(p-1)\int_{0}^{1}|\theta\xi+(1-\theta)\eta|^{p-2}|\xi-\eta|^{2}\theta d\theta\\
  &\qquad\quad+(p-2)\int_{0}^{1}|\theta\xi+(1-\theta)\eta|^{p-4}(|\xi|^2|\eta|^{2}-\Re^{2}(\xi\cdot\bar{\eta}))\theta d\theta,
\end{aligned}
\end{equation*}
and if $p\geq2$,
$$  \frac{1}{p} |\eta|^p + \frac{1}{p'}  |\xi|^p - |\xi|^{p-2} \Re(\xi\cdot\bar{\eta})\ge (p-1)\int_{0}^{1}|\theta\xi+(1-\theta)\eta|^{p-2}|\xi-\eta|^{2}\theta d\theta.$$
The above inequality becomes an equality when $\xi=\lambda\eta$ with $\lambda\in \mathbb R$.

This finishes the proof of Theorem \ref{exrep}.
\end{proof}

\begin{proof}[Proof of Theorem \ref{solution}]
If $1<p<2$, taking $x=ky$ with $k\to-\infty$ in \eqref{cplinear}, 
we obtain $c_p=0$.
Suppose $p\ge2$ now. 
From \eqref{estimation} and the reasoning in Remark \ref{rem:lambda}, we conclude that the sharp constant $c_p$ is attained only if $x=ky$ with $k\in\mathbb{R}$.
Thus we obtain
\begin{equation}\label{cplinear}
\begin{aligned}
c_p=\inf_{x, y \in \mathbb{C}^d}\frac{C_{p}(x,y)}{|y|^p}
&=p(p-1)\inf_{x=ky }\int_{0}^{1}\left|\theta \frac{y}{|y|}-\frac{x}{|y|}\right|^{p-2}\theta d\theta\\
&=p(p-1)\inf_{k\in\mathbb R }\int_{0}^{1}\left|\theta-k\right|^{p-2}\theta d\theta.
\end{aligned}
\end{equation}
Recall that we can omit the discussion for $k\in\{0,1\}$.

For the case $0<k<1$, we obtain
\begin{equation*}
\begin{aligned}
&p(p-1)\int_{0}^{1}\left|\theta-k\right|^{p-2}\theta d\theta\\
&\qquad=p(p-1)\int_{0}^{k}(k-\theta)^{p-2}\theta d\theta+p(p-1)\int_{k}^{1}(\theta-k)^{p-2}\theta d\theta\\
&\qquad=k^p+pk(1-k)^{p-1}+(p-1)(1-k)^p.
\end{aligned}
\end{equation*}
Introduce 
$$
f(x):=x^p+px(1-x)^{p-1}+(p-1)(1-x)^p,
$$
thus
$$
f'(x)=p\left(x^{p-1}+(1-x)^{p-1}-(p-1)(1-x)^{p-2}\right).
$$
Note that $f'(0)=p(2-p)\le 0$ and $f'(1)=p>0$.
Suppose $f'(x_0)=0$ and $x_0=\frac{r_{0}}{r_{0}+1}$ with $r_{0}>0$, then the equation $f'(x_0)=0$ can be rewritten as
\begin{equation}\label{root}
r_{0}^{p-1}-(p-1)r_{0}-(p-2)=0.
\end{equation}
Note that $$(r^{p-1}-(p-1)r-(p-2))'=(p-1)(r^{p-2}-1),$$
which alternates the sign passing $r=1$. 
Moreover, $$(r^{p-1}-(p-1)r-(p-2))|_{r=0}=-(p-2)\le0$$
and $$(r^{p-1}-(p-1)r-(p-2))|_{r=+\infty}=+\infty.$$
We then conclude that $r_{0}$ is unique for $f'(\frac{r_{0}}{r_{0}+1})=0$. 
For the case $k>1$, we obtain
$$
p(p-1)\int_{0}^{1}(k-\theta)^{p-2}\theta d\theta>p(p-1)\int_{0}^{1}(1-\theta)^{p-2}\theta d\theta=1=f(1).
$$
For the case $k<0$, we obtain
$$
p(p-1)\int_{0}^{1}(\theta-k)^{p-2}\theta d\theta>p(p-1)\int_{0}^{1}\theta^{p-1} d\theta=p-1=f(0).
$$
Thus $c_p=f(x_{0})$. 

Now we check $f(x_{0})$ is always positive. Suppose that $f(x_{0})=0$, we obtain 
\begin{equation}\label{positive}
\left(\frac{r_0}{r_0+1}\right)^p+p\frac{r_0}{r_0+1}\left(\frac{1}{r_0+1}\right)^{p-1}+(p-1)\left(\frac{1}{r_0+1}\right)^p=0.
\end{equation}
Using \eqref{root}, \eqref{positive} can be rewritten as
$$
(p-1)\left(r_{0}^{2}+2r_{0}+1\right)=0,
$$
which is a contradiction to $r_0>0$. 
This proves Theorem \ref{solution}.
\end{proof}

\bigskip

\section*{\textbf{Compliance with ethical standards}}

\bigskip

\textbf{Conflict of interest} The authors have no known competing financial interests or personal relationships that could have appeared to influence this reported work.

\bigskip

\textbf{Availability of data and material} Not applicable.

\bigskip

\bibliographystyle{alpha}
\bibliography{remainder}

\end{document}